\documentclass{article}[12pt]
\usepackage[french,english]{babel}
\usepackage{graphicx}
\usepackage{amsmath}
\usepackage{amsfonts}
\usepackage{amssymb}
\usepackage{amsthm}
\usepackage{hyperref}
\usepackage[T1]{fontenc}
\usepackage{color}
\usepackage{mathtools}
\usepackage{comment}

\input xy
\xyoption{all}

\numberwithin{equation}{section}

\newtheorem{thm}{Th\'eor\`eme}
\newtheorem{cor}[thm]{Corollaire}
\newtheorem{lem}[thm]{Lemme}
\newtheorem{pro}[thm]{Proposition}

\theoremstyle{definition}
\newtheorem{defi}[thm]{D\'efinition}

\theoremstyle{remark}
\newtheorem*{rem}{Remarque}

\newcommand{\asd}[5]{%
\setbox1=\hbox{\ensuremath{^{#1}}}%
\setbox2=\hbox{\ensuremath{_{#2}}}%
\setbox5=\hbox{\ensuremath{#5}}%
\hspace{\ifnum\wd1>\wd2\wd1\else\wd2\fi}%
\ensuremath{\copy5^{\hspace{-\wd1}\hspace{-\wd5}#1\hspace{\wd5}#3}%
_{\hspace{-\wd2}\hspace{-\wd5}#2\hspace{\wd5}#4}%
}}

\usepackage[OT2,T1]{fontenc}
\DeclareSymbolFont{cyrletters}{OT2}{wncyr}{m}{n}
\DeclareMathSymbol{\Sha}{\mathalpha}{cyrletters}{"58}
\DeclareMathSymbol{\Brusse}{\mathalpha}{cyrletters}{"42}

\newcommand{\n}{\mathbb{N}}
\newcommand{\z}{\mathbb{Z}}
\newcommand{\q}{\mathbb{Q}}

\renewcommand{\ker}{\mathrm{Ker }}
\newcommand{\im}{\mathrm{Im}}
\renewcommand{\hom}{\mathrm{Hom}}

\newcommand{\aut}{\mathrm{Aut}}
\newcommand{\out}{\mathrm{Out}}
\renewcommand{\int}{\mathrm{Int}}
\newcommand{\saut}{\mathrm{SAut}}

\newcommand{\nr}{\mathrm{nr}}

\newcommand{\al}{\mathrm{al}}

\newcommand{\brnral}{\mathrm{Br}_{\nr,\al}}
\newcommand{\gal}{\mathrm{Gal}}
\newcommand{\sln}{\mathrm{SL}_n}

\title{Approximation faible et principe de Hasse pour des espaces homog\`enes \`a stabilisateur fini r\'esoluble}
\author{Giancarlo Lucchini Arteche\\[5mm]
{\it\small D\'epartement de math\'ematiques, universit\'e Paris-Sud}\\
{\it\small b\^atiment 425, 91405 Orsay cedex, France}\\
{\small giancarlo.lucchini@math.u-psud.fr}
}

\date{}

\begin{document}

\selectlanguage{french}

\maketitle

\selectlanguage{english}
\begin{abstract}
{\bf Weak approximation and Hasse principle for homogeneous spaces with finite solvable stabilizer.} Let $K$ be a global field and $G$ a finite solvable $K$-group.
Under certain hypotheses concerning the extension splitting $G$, we show that the homogeneous space $V=G'/G$ with $G'$ a semi-simple simply connected $K$-group
has the weak approximation property. We use a more precise version of this result to prove the Hasse principle for homogeneous spaces $X$ under a semi-simple
simply connected $K$-group $G'$ with finite solvable geometric stabilizer $\bar G$, under certain hypotheses concerning the $K$-kernel (or $K$-lien)
$(\bar G,\kappa)$ defined by $X$.
{\bf MSC classes (2010):} 11E72, 14M17, 14G05.
\end{abstract}

\selectlanguage{french}
\begin{abstract}
Soit $K$ un corps global et $G$ un $K$-groupe fini r\'esoluble. Sous certaines hypoth\`eses sur une extension d\'eployant $G$, on d\'emontre que
l'espace homog\`ene $V:=G\backslash G'$ avec $G'$ un $K$-groupe semi-simple simplement connexe v\'erifie l'approximation faible. On utilise une version plus
pr\'ecise de ce r\'esultat pour d\'emontrer le principe de Hasse pour des espaces homog\`enes $X$ sous un $K$-groupe $G'$ semi-simple
simplement connexe \`a stabilisateur g\'eom\'etrique $\bar G$ fini et r\'esoluble, sous certaines hypoth\`eses sur le $K$-lien $(\bar G,\kappa)$
d\'efini par $X$.\\
{\bf Mots cl\'es :} espace homog\`ene, approximation faible, principe de Hasse, cohomologie galoisienne.
\end{abstract}

\section{Introduction}
Soit $K$ un corps global, i.e. un corps de nombres ou le corps de fonctions d'une courbe $C$ sur un corps fini. Dans le cadre des espaces homog\`enes de
groupes lin\'eaires connexes, l'approximation faible et le principe de Hasse sont deux propri\'et\'es qui ont \'et\'e v\'erifi\'ees, par exemple, pour
les espaces principaux homog\`enes sous un groupe semi-simple et simplement connexe, et cela depuis d\'ej\`a presque un demi-si\`ecle. Pour les espaces
homog\`enes \`a stabilisateur non trivial, ou m\^eme pour les espaces principaux homog\`enes sous les tores, on conna\^\i t cependant des exemples ne
v\'erifiant pas le principe de Hasse, ainsi que des exemples ayant des points rationnels mais ne v\'erifiant pas l'approximation faible.

Or, depuis la mise en \'evidence de l'obstruction de Brauer-Manin pour le principe de Hasse et pour l'approximation faible, introduites respectivement par Manin
et par Colliot-Th\'el\`ene et Sansuc, on ne conna\^\i t pas d'espace homog\`ene dont le d\'efaut de ces propri\'et\'es ne soit pas
expliqu\'e par ces obstructions. La question
naturelle \`a se poser alors est de savoir si cette obstruction est la seule dans les deux cas. Dans \cite{Borovoi96}, Borovoi a d\'emontr\'e que
l'obstruction de Brauer-Manin au principe de Hasse et \`a l'approximation faible est la seule pour les espaces homog\`enes sous un groupe lin\'eaire
connexe \`a stabilisateur connexe, ou encore \`a stabilisateur ab\'elien en supposant que la partie semi-simple du groupe ambiant est simplement connexe.
La question g\'en\'erale reste pourtant totalement ouverte.

Si l'on consid\`ere le cas ``oppos\'e'' de celui des stabilisateurs connexes, i.e. celui des
stabilisateurs finis, on ne conna\^\i t que quelques cas particuliers, comme le cas ab\'elien chez Borovoi d\'ej\`a mentionn\'e, ou les travaux de Harari
et Demarche sur l'approximation faible pour des quotients de la forme $G\backslash\sln$ pour certains $G$ finis constants et non ab\'eliens (c.f.
\cite{HarariBulletinSMF} et \cite{Demarche}). Le r\'esultat principal dans cette direction est cependant celui de Neukirch dans \cite{NeukirchSolvable},
lequel a inspir\'e la r\'edaction de ce texte.\\

Lorsqu'on veut consid\'erer des stabilisateurs finis, on peut le faire sous diff\'erents groupes ambiants. Or, du moins pour l'approximation faible,
la question est triviale lorsqu'on consid\`ere un groupe ambiant unipotent et la r\'eponse est d\'ej\`a connue pour un groupe ambiant torique (car la
question se ram\`ene trivialement aux espaces principaux homog\`enes), le cas int\'eressant \'etant alors celui des groupes semi-simples.
De plus, si l'on consid\`ere un $K$-groupe fini $G$ plong\'e dans un groupe $G'$ semi-simple et si l'on note $\rho:\tilde G'\to G'$ le rev\^etement
universel de $G'$, il est ais\'e de voir que l'on a un isomorphisme canonique entre les vari\'et\'es $\tilde V:=\rho^{-1} (G)\backslash\tilde G'$ et
$V:=G\backslash G'$ induit par $\rho$. Le noyau de $\rho$ \'etant fini et central, le groupe $\rho^{-1}(G)$ est alors un groupe fini et une extension
centrale de $G$. Ainsi, on voit que l'on peut se restreindre \`a \'etudier le cas des groupes ambiants semi-simples simplement connexes comme il a
d\'ej\`a \'et\'e fait par Borovoi, Demarche et Harari dans leurs travaux.

Le but de ce texte est donc de d\'emontrer un r\'esultat d'approximation faible pour des espaces homog\`enes $V:=G\backslash G'$ avec $G'$ semi-simple
simplement connexe et $G$ fini et r\'esoluble (i.e. $G(\bar K)$ \'etant un groupe r\'esoluble). On d\'emontre notamment le r\'esultat suivant :

\begin{thm}[Cons\'equence du th\'eor\`eme \ref{theoreme AF pour les groupes resolubles}]\label{theoreme intro}
Soient $K$ un corps global et $G$ un $K$-groupe fini r\'esoluble d'exposant $n$ premier \`a la caract\'eristique de $K$. On suppose qu'il existe une
extension finie galoisienne $K'/K$ d\'eployant $G$ et telle que, si l'on note $m(K')$ le cardinal de l'ensemble des racines de l'unit\'e contenues dans $K'$, on a
$(m(K'),n)=1$. Soit $V=G\backslash G'$ pour un $K$-plongement $G\to G'$ avec $G'$ un $K$-groupe semi-simple simplement connexe. Alors $V$ v\'erifie
l'approximation faible.
\end{thm}

Ce th\'eor\`eme am\'eliore le r\'esultat de Neukirch dans \cite{NeukirchSolvable}, corollaire de ses travaux sur le probl\`eme du plongement,
qui montre le cas particulier du th\'eor\`eme ci-dessus pour les groupes constants. Il simplifie aussi la d\'emonstration de ce cas particulier car, le
probl\`eme du plongement \'etant d'une nature diff\'erente de celle de l'approximation faible, la preuve de Neukirch devient par cons\'equent assez
complexe.
Il est aussi int\'eressant de remarquer que ce r\'esultat entra\^\i ne la trivialit\'e de l'obstruction de Brauer-Manin \`a l'approximation faible pour
ces espaces. En fait, on peut m\^eme d\'emontrer, en utilisant les formules dans \cite{GLABrnral}, que ces espaces
homog\`enes ont un groupe de Brauer non ramifi\'e alg\'ebrique $\brnral V$ trivial (cf. aussi \cite[Corollaire 5.7]{ColliotBrnr} pour le cas particulier
o\`u le stabilisateur $G$ est constant).

On voit aussi que les hypoth\`eses faites sur les racines de l'unit\'e contenues dans le corps $K'$ imposent que $G$ soit toujours un groupe d'ordre
impair, car tout corps global de caract\'eristique diff\'erente de $2$ contient $\mu_2=\{1,-1\}$, et en caract\'eristique $2$ on \'evite de prendre des
groupes d'ordre pair. On en d\'eduit que le fait de demander que $G$ soit r\'esoluble est en un certain sens superflu,
vu que l'on sait d'apr\`es le th\'eor\`eme de Feit et Thompson que tout groupe fini d'ordre impair est r\'esoluble. Ainsi, le r\'esultat de Neukirch
entra\^\i nait d\'ej\`a l'approximation faible pour tout $\q$-groupe fini $G$ constant d'ordre impair, et la g\'en\'eralisation que l'on pr\'esente
ici entra\^\i ne l'approximation faible, par exemple, d\`es que $G$ est d'ordre impair et d\'eploy\'e par une extension $K'/K$ ayant une place r\'eelle,
car dans ce cas on a bien $\mu(K')=\mu_2$.\\

On rappelle que pour toute $K$-vari\'et\'e $V$ on a un plongement diagonal
\[V(K)\hookrightarrow \prod_{v\in\Omega_K}V(K_v),\]
de ses $K$-points dans le produit sur $\Omega_K$ de ses $K_v$-points, o\`u $\Omega_K$ est l'ensemble des places de $K$ et $K_v$ est le
compl\'et\'e de $K$ par rapport \`a la place $v$. L'ensemble $\prod_{v\in\Omega_K}V(K_v)$ peut \^etre muni de la topologie produit, o\`u chacun des $V(K_v)$
est muni de la topologie $v$-adique. On dit alors que $V$ v\'erifie l'approximation faible si l'image de l'ensemble $V(K)$ par ce plongement est dense dans l'ensemble
$\prod_{v\in\Omega_K}V(K_v)$. Dans le cas o\`u $V(K)\neq\emptyset$, cela revient \`a demander que, pour tout ensemble fini $S$ de places de $K$, l'ensemble
$V(K)$ soit dense dans le produit $\prod_{v\in S}V(K_v)$.
Or, dans le cadre des espaces homog\`enes sous un groupe $G'$ semi-simple simplement connexe cette propri\'et\'e admet une traduction cohomologique :
l'approximation faible pour la vari\'et\'e $V=G\backslash G'$ consid\'er\'ee dans le th\'eor\`eme \ref{theoreme intro} est en fait \'equivalente \`a la
surjectivit\'e de l'application
\begin{equation}\label{equation rappel approximation faible}
H^1(K,G)\to \prod_{v\in S} H^1(K_v,G),
\end{equation}
pour tout ensemble fini $S$ de places de $K$. En effet, il suffit de consid\'erer
le diagramme commutatif \`a lignes exactes suivant.
\[\xymatrix{
G'(K) \ar[r] \ar[d] & V (K) \ar[r]^{\delta} \ar[d] & H^1(K,G) \ar[r] \ar[d] & H^1(K,G') \ar[d] \\
\displaystyle{\prod_{v\in S}}G'(K_v) \ar[r] & \displaystyle{\prod_{v\in S}}V(K_v) \ar[r]^{\delta} & \displaystyle{\prod_{v\in S}}H^1(K_v,G) \ar[r] &
\displaystyle{\prod_{v\in S}}H^1(K_v,G')
}\]
et de rappeler que pour toute place non archim\'edienne on a $H^1(K_v,G')=1$ (cf. \cite{KneserH1LocalI}, \cite{KneserH1LocalII}, \cite{BruhatTits})
et que l'on a une bijection
\[H^1(K,G')\xrightarrow{\sim} \prod_{v\in\Omega_K}H^1(K_v,G'),\]
ce qui donne la nullit\'e de $H^1(K,G')$ dans le cas de caract\'eristique positive (cf.
\cite{HarderHasseCarp}), ainsi qu'une bijection
\begin{equation}\label{equation PH SS SC}
H^1(K,G')\xrightarrow{\sim} \prod_{v\in\Omega_\infty}H^1(K_v,G'),
\end{equation}
dans le cas d'un corps de nombres, o\`u $\Omega_\infty$ repr\'esente l'ensemble fini des places archim\'ediennes (cf. \cite{HarderHasseCar0I}, \cite{HarderHasseCar0II},
\cite{KneserHasseCar0}, \cite{ChernousovE8}). Prenant en compte que la vari\'et\'e $G'$ v\'erifie l'approximation faible (cf. \cite{KneserAF}) et que
les application $\delta$ sont continues pour les topologies correspondantes (discr\`ete pour $V(K)$ et pour les $H^1$, $v$-adique pour $V(K_v)$),
il est facile d'en d\'eduire l'\'equivalence entre l'approximation aux places $v\in S$ pour $V$ et la surjectivit\'e de
\eqref{equation rappel approximation faible}.\\

On voit alors qu'il s'agit ici de regarder la cohomologie galoisienne non ab\'elienne des groupes finis r\'esolubles. Les r\'esultats que l'on montrera
par la suite se pr\^etent donc \`a des d\'emonstrations par d\'evissage, i.e. il s'agira de casser le groupe $G$ en des morceaux plus simples.
Si l'on note alors $\Gamma_K=\gal(\bar K/K)$ pour une cl\^oture s\'eparable $\bar K$ de $K$, l'argument de d\'evissage nous m\^ene \`a l'\'etude
des $\Gamma_K$-modules simples dont l'ordre divise le cardinal du groupe. Il faudra pourtant demander un peu plus
que l'approximation faible \`a ces groupes pour pouvoir recoller les morceaux, l'id\'ee \`a la base \'etant de ``cacher la ramification dans de
bonnes places''. Dans la section \ref{section AF modules simples} on rappelle ce qui est connu dans cette direction, et qui est essentiellement
d\^u \`a Neukirch (cf. \cite{NeukirchSolvable} et \cite[Chapter IX]{NSW}), et on montre quelques r\'esultats suppl\'ementaires qui nous seront utiles
pour d\'emontrer le th\'eor\`eme principal. La section \ref{section AF} est consacr\'ee
\`a l'\'enonc\'e et la d\'emonstration du th\'eor\`eme principal. Ce r\'esultat nous permet par la suite de d\'emontrer un r\'esultat autour du principe
de Hasse pour des espaces homog\`enes \`a stabilisateur fini r\'esoluble dans la section \ref{section PH}.

\paragraph*{Notations}
Dans toute la suite de ce texte, on notera $\Gamma_{L/K}:=\gal(L/K)$ pour toute extension galoisienne $L/K$ et on notera $\Gamma_K$ le groupe de
Galois absolu de tout corps $K$, i.e. le groupe $\gal(\bar K/K)$ avec $\bar K$ une cl\^oture s\'eparable de $K$. Les groupes (ou ensembles) de cohomologie
galoisienne associ\'es \`a ces groupes de Galois seront not\'es respectivement $H^i(L/K,\cdot)$ et $H^i(K,\cdot)$. Le corps $K$ est toujours un corps
global, i.e. un corps de nombres ou le corps de fonction d'une courbe sur un corps fini, dont l'ensemble des places est not\'e $\Omega_K$ et dont
les diff\'erents compl\'et\'es sont not\'es $K_v$ pour $v\in\Omega_K$. On \'ecrit $\mu(K)$ pour d\'esigner l'ensemble des racines de l'unit\'e
contenues dans le corps $K$ et $m(K)$ pour d\'esigner son cardinal. Enfin, pour tout entier $n\geq 2$ premier \`a la caract\'eristique de $K$,
on note $\mu_n$ le sous-groupe de $\bar K^*$ des racines $n$-i\`emes de l'unit\'e et $\zeta_n\in\mu_n$ un g\'en\'erateur de ce groupe.

On fera aussi un abus de notation en identifiant tout $K$-groupe $G$ fini d'ordre premier \`a la caract\'eristique de $K$ avec le $\Gamma_K$-groupe (ou module)
$G(\bar K)$ correspondant. Ainsi, on se permettra de noter $b\in G$ pour tout point g\'eom\'etrique $b\in G(\bar K)$ et aussi, pour tout sous-quotient
$\Delta$ de $\Gamma_K$ dont l'action sur $G(\bar K)$ est bien d\'efinie, on notera
\[H^1(\Delta,G):=H^1(\Delta,G(\bar K))\quad\text{et}\quad\hom(\Delta,G):=\hom_{\text{cont}}(\Delta,G(\bar K)),\]
Ceci s'appliquera notamment, dans le cas o\`u $K$ est un corps global, pour $\Delta=D_v$, o\`u $D_v$ correspond au quotient mod\'er\'ement
ramifi\'e du groupe de d\'ecomposition d'une place $v$ de $K$ et $G$ est un groupe d\'eploy\'e par l'extension maximale mod\'er\'ement ramifi\'ee.
On remarque aussi que, puisque $G$ est fini, on a $G(\bar K)=G(\bar L)$ pour tout corps $L$ contenant $K$, ce qui permet par exemple l'identification de
$H^1(\Gamma_{K_v},G(\bar K))$ avec $H^1(K_v,G)$. Ces identifications seront toujours sous-entendues dans la suite.

Pour tout $\Gamma_K$-module fini $A$ d'exposant $n$, on notera $A^*$ le $\Gamma_K$-module $\hom(A,\mu_n)$. Pour un tel module,
les groupes de Tate-Shafarevich $\Sha^i(K,A)$ sont d\'efinis, pour $i=1,2$ en ce qui nous concerne, comme
\[\Sha^i(K,A):=\ker\left[H^i(K,A)\to\prod_{v\in\Omega_K}H^i(K_v,A)\right].\]
On rappelle aussi que la dualit\'e de Poitou-Tate nous donne une dualit\'e parfaite entre $\Sha^2(K,A)$ et $\Sha^1(K,A^*)$.

Enfin, puisque dans la suite on aura toujours recours \`a diff\'erents ensembles infinis de places de $K$ qui \'evitent d'autres ensembles finis de places,
on utilisera l'expression ``presque toute place {\tt xxxxx}'' pour signifier ``toute place {\tt xxxxx} \`a l'exception d'une quantit\'e finie''.

\section{Pr\'eliminaires}\label{section AF modules simples}
Soit $A$ un $\Gamma_K$-module fini d'ordre premier \`a la caract\'eristique de $K$ (lequel correspond aux points g\'eom\'etriques d'un $K$-groupe
alg\'ebrique fini et ab\'elien). Comme il a \'et\'e \'evoqu\'e dans l'introduction, on entend par approximation faible pour $A$ la surjectivit\'e
de l'application
\[H^1(K,A)\to \prod_{v\in S} H^1(K_v,A),\]
pour tout ensemble fini $S$ de places de $K$. Neukirch a montr\'e que, sous certaines hypoth\`eses, on peut non seulement demander l'existence d'une
classe globale $\alpha\in H^1(K,A)$ relevant une famille locale donn\'ee $(\beta_v)_{v\in S}\in\prod_SH^1(K_v,A)$, mais aussi demander qu'une telle
classe $\alpha$ ait des bonnes propri\'et\'es de ramification, dans un sens qui sera explicit\'e par la suite. Ce r\'esultat d'approximation
``mieux que faible'' a \'et\'e d\'emontr\'e par Neukirch dans \cite[Theorem 1]{NeukirchSolvable} pour les $\Gamma_K$-modules simples sur un corps
des nombres, et une version pour des corps globaux quelconques peut \^etre trouv\'ee dans \cite[Theorem 9.3.3]{NSW}. On \'enonce dans cette section
un corollaire de la version g\'en\'erale que l'on utilisera pour montrer le th\'eor\`eme principal. Pour ce faire, on rappelle quelques d\'efinitions.

\begin{defi}
Soit $K_v$ le compl\'et\'e d'un corps global en une place finie et soit $G$ un ${K_v}$-groupe (resp. un $\Gamma_{K_v}$-module). Une classe
$\alpha\in H^1(K_v,G)$ est dite :
\begin{itemize}
\item \emph{cyclique} si elle est d\'eploy\'ee par une extension cyclique $L_v/K_v$, i.e. si sa
restriction \`a $H^1(L_v,G)$ est triviale ou, de fa\c con \'equivalente, si elle provient de l'ensemble (resp. groupe) $H^1(L_v/K_v,G^{L_v})$
par inflation ;
\item \emph{non ramifi\'ee} si elle est d\'eploy\'ee par l'extension non ramifi\'ee maximale $K^\nr_v$ de $K_v$, i.e. si elle provient de l'ensemble
(resp. groupe) $H^1(K_v^\nr/K_v,G^{\Gamma_{K_v^{\nr}}})$ par inflation ;
\item \emph{ramifi\'ee} si elle n'est pas non ramifi\'ee ;
\item \emph{totalement ramifi\'ee} si elle est d\'eploy\'ee par une extension totalement ramifi\'ee $L_v/K_v$.
\end{itemize}
\end{defi}

\begin{rem}\label{remarque nr et tr}
Une classe non ramifi\'ee est toujours cyclique. En effet, toute classe $\alpha\in H^1(K_v,G)$ est d\'eploy\'ee par une extension
finie, d'o\`u l'on sait qu'une classe $\alpha$ provenant de $H^1(K_v^\nr/K_v,G)$ provient en fait de $H^1(L_v/K_v,G)$, o\`u $\Gamma_{L_v/K_v}$ correspond
\`a un quotient fini de $\Gamma_{K^\nr_v/K_v}\cong\hat\z$ et il est donc un groupe cyclique.

De plus, lorsque $G$ est un groupe constant, il est facile de voir qu'une classe non triviale et totalement ramifi\'ee est forc\'ement ramifi\'ee.
\end{rem}

Avec ces d\'efinitions, on peut \'enoncer l'analogue du r\'esultat de Neukirch. Dans toute la suite du texte, $\ell$ d\'esigne un nombre premier diff\'erent de $p$.

\begin{thm}\label{theoreme 1 de Neukirch}
Soit $K$ un corps global de caract\'eristique $p\geq 0$ et soit $A$ un $\Gamma_K$-module fini simple de $\ell$-torsion pour $\ell\neq p$. On suppose qu'il existe une
extension finie galoisienne $K'/K$ d\'eployant $A$ telle que $\zeta_\ell\not\in K'$. Soit $L/K'$ une extension cyclotomique (i.e. engendr\'ee par des racines de
l'unit\'e) contenant $\zeta_\ell$. Soit $S$ un ensemble fini de places de $K$ contenant les places archim\'ediennes et les places divisant $\ell$ et
soit enfin $P$ l'ensemble des places de $K$ totalement d\'ecompos\'ees pour l'extension $L/K$ n'appartenant pas \`a $S$.

Alors pour toute famille $(\beta_v)_{v\in S}\in\prod_SH^1(K_v,A)$, il existe un \'el\'ement $\alpha\in H^1(K,A)$ tel que, si l'on note $\alpha_v$
l'image de $\alpha$ dans $H^1(K_v,A)$, on a
\begin{itemize}
\item pour toute place $v\in S$, $\alpha_v=\beta_v$ ;
\item pour toute place $v\not\in S$, $\alpha_v$ est cyclique et, si $\alpha_v$ est ramifi\'e, alors $v\in P$.
\end{itemize}
\end{thm}

\begin{rem}
Dans l'article \cite{NeukirchSolvable}, le r\'esultat de Neukirch n'est valable qu'en caract\'eristique $0$ et pour $L=K'(\zeta_\ell)$.
\end{rem}

\begin{proof}
Le r\'esultat d\'ecoule facilement de \cite[Theorem 9.3.3]{NSW} de la fa\c con suivante. Soit $\Delta$ le sous-groupe de torsion premi\`ere \`a $\ell$ du groupe
ab\'elien $\Gamma_{L/K'}$ et soit $K'_\ell:=L^\Delta\subset L$. On voit alors que $[L:K'_\ell]$ est premier \`a $\ell$ et que $\zeta_\ell\not\in K'_\ell$ car
$[K'_\ell:K']$ est une puissance de $\ell$ et $\zeta_\ell\not\in K'$ (et l'extension $K'(\zeta_\ell)/K'$ est de degr\'e premier \`a $\ell$). Les donn\'ees
$(\ell,K,K'_\ell,L,\Omega_K,S)$ ci-dessus v\'erifient alors clairement les hypoth\`eses demand\'ees par \cite[Theorem 9.3.3]{NSW} pour, dans les notations
de \cite{NSW}, $(p,k,K,\Omega,S,T)$.
\end{proof}

La m\^eme d\'emonstration, avec les m\^emes notations mais en mettant \cite[Theorem 9.1.15(iii)]{NSW} \`a la place de \cite[Theorem 9.3.3]{NSW} et $P$ \`a la
place de $T$, s'applique par ailleurs au lemme ci-dessous, dont l'analogue chez Neukirch (cf. \cite[Lemma 2]{NeukirchSolvable}) est cl\'e dans la d\'emonstration
de son th\'eor\`eme :

\begin{lem}\label{lemme 2 de Neukirch}
Soit $K$ un corps global de caract\'eristique $p\geq 0$ et soit $A$ un $\Gamma_K$-module fini simple de $\ell$-torsion pour $\ell\neq p$. On suppose qu'il existe une
extension finie galoisienne $K'/K$ d\'eployant $A$ telle que $\zeta_\ell\not\in K'$. Soit $L/K'$  une extension cyclotomique contenant $\zeta_\ell$. Soit enfin $P$ un
ensemble de places de $K$ contenant presque toutes les places totalement d\'ecompos\'ees pour l'extension $L/K$. Alors le morphisme
\[H^1(K,A^*)\to\prod_{v\in P}H^1(K_v,A^*),\]
est injectif, o\`u $A^*:=\hom(A,\mu_{\ell})$.
\end{lem}

\begin{rem}
On notera en passant que la simplicit\'e de $A$ en tant que $\Gamma_K$-module est \'equivalente \`a celle de $A^*$.
\end{rem}

Ce lemme admet la g\'en\'eralisation suivante aux $\Gamma_K$-modules finis pas n\'ecessairement simples dont la preuve se ram\`ene par un
argument de d\'evissage au cas ci-dessus.

\begin{pro}\label{proposition injectivite en P pour les groupes abeliens}
Soit $K$ un corps global de caract\'eristique $p\geq 0$ et soit $A$ un $\Gamma_K$-module fini d'exposant $n$ premier \`a $p$. On suppose qu'il existe une
extension finie galoisienne $K'/K$ d\'eployant $A$ telle que, si l'on note $\mu(K')$ l'ensemble des racines de l'unit\'e contenues dans $K'$ et $m(K')$ son cardinal,
alors $(m(K'),n)=1$. Soit $L/K'$ une extension cyclotomique contenant $\zeta_n$. Soit enfin $P$ un ensemble de places de $K$ contenant presque toutes
les places totalement d\'ecompos\'ees pour l'extension $L/K$. Alors le morphisme
\[H^1(K,A^*)\to\prod_{v\in P}H^1(K_v,A^*),\]
est injectif, o\`u $A^*=\hom(A,\mu_n)$.
\end{pro}

\begin{proof}
Il est \'evident d'abord que l'on peut se restreindre au cas o\`u $P$ n'est compos\'e que de places totalement d\'ecompos\'ees pour l'extension $L/K$.
La preuve se fait alors par r\'ecurrence sur le cardinal de $A$, l'injectivit\'e \'etant v\'erifi\'ee dans le cas o\`u $A$ est simple d'apr\`es le lemme
\ref{lemme 2 de Neukirch}. Si $A$ n'est pas simple, on sait qu'il existe un sous-module $B$ non trivial induisant un quotient $C=A/B$ et l'on peut
supposer que $B$ et $C$ v\'erifient l'\'enonc\'e. On consid\`ere alors le diagramme commutatif \`a lignes exactes suivant :
\[\xymatrix{
H^1(K,C^*) \ar[d] \ar[r] & H^1(K,A^*) \ar[d] \ar[r] & H^1(K,B^*) \ar[d] \\
\displaystyle{\prod_{v\in P} H^1(K_v,C^*)} \ar[r] & \displaystyle{\prod_{v\in P} H^1(K_v,A^*)} \ar[r] & \displaystyle{\prod_{v\in P} H^1(K_v,B^*)} 
}\]
Soit alors $\alpha\in H^1(K,A^*)$ d'image nulle dans $\prod_P H^1(K_v,A^*)$. On voit alors que son image dans $\prod_P H^1(K_v,B^*)$ est nulle
et par cons\'equent son image dans $H^1(K,B^*)$ est nulle aussi d'apr\`es l'injectivit\'e pour $B^*$. L'\'el\'ement $\alpha$ provient donc de $H^1(K,C^*)$.
Or, puisque pour toute place $v\in P$ on a que $\zeta_n\in K_v$, l'action de $\Gamma_{K_v}$ sur les modules $A$, $B$, $C$ et $\mu_n$ est triviale, ce qui
nous dit que les $\Gamma_{K_v}$-modules $A^*$, $B^*$ et $C^*$ sont d\'eploy\'es, d'o\`u l'on voit qu'on a la suite exacte
\[A^*\to B^* \to H^1(K_v,C^*)\to H^1(K_v,A^*),\]
\`a partir de laquelle on d\'eduit que l'application $H^1(K_v,C^*)\to H^1(K_v,A^*)$ est injective pour toute place $v\in P$. On voit alors, en regardant
une pr\'eimage $\gamma$ de $\alpha$ dans $H^1(K,C^*)$, que ses localis\'es $\gamma_v\in H^1(K_v,C^*)$ sont triviaux pour toute place $v\in P$ car leur image
$\alpha_v\in H^1(K_v,A^*)$ l'est aussi, d'o\`u l'on voit que $\gamma=\alpha=0$ d'apr\`es l'injectivit\'e pour $C^*$, ce qui conclut.
\end{proof}

La proposition \ref{proposition injectivite en P pour les groupes abeliens} a une cons\'equence int\'eressante, laquelle avait
d\'ej\`a \'et\'e remarqu\'ee par Neukirch pour le cas des $\Gamma_K$-modules simples (cf. \cite[Theorem 2]{NeukirchSolvable}), et qui nous sera utile
par la suite. 

\begin{cor}\label{corollaire nullite du sha2 pour les groupes abeliens}
Soient $K$ un corps global et $A$ un $K$-groupe fini ab\'elien d'exposant $n$ premier \`a la caract\'eristique de $K$. On suppose qu'il existe une
extension finie galoisienne $K'/K$ d\'eployant $A$ telle que $(m(K'),n)=1$. Alors on a $\Sha^2(K,A)=0$.
\end{cor}

\begin{proof}
D'apr\`es la proposition \ref{proposition injectivite en P pour les groupes abeliens}, l'application
\[H^1(K,A^*)\to\prod_{v\in P}H^1(K_v,A^*),\]
est injective, d'o\`u $\Sha^1(K,A^*)=0$. On voit aussit\^ot par la dualit\'e de Poitou-Tate que $\Sha^2(K,A)=0$.
\end{proof}

Pour terminer cette section, on donne un lemme qui justifiera par la suite le fait de forcer la ramification \`a tomber dans les places
$v\in\Omega_K$ totalement d\'ecompos\'ees pour une extension $L/K$ comme celle de l'\'enonc\'e du th\'eor\`eme \ref{theoreme 1 de Neukirch}.

\begin{lem}\label{lemme de relevement local}
Soit $K_v$ un corps local de caract\'eristique r\'esiduelle $p>0$, soit $G$ un $K_v$-groupe constant d'exposant $n$ et soit $H$ un $K_v$-groupe
quotient de $G$ (qui est par cons\'equent constant). On suppose que $n$ est premier \`a $p$ et que $\zeta_n\in K_v$. Soit $\alpha\in H^1(K_v,H)$
une classe totalement ramifi\'ee et cyclique. Alors il existe une classe $\beta\in H^1(K_v,G)$ totalement ramifi\'ee et cyclique relevant $\alpha$,
i.e. $\alpha$ est l'image de $\beta$ par l'application $H^1(K_v,G)\to H^1(K_v,H)$.
\end{lem}

\begin{proof}
Soit $S_v$ le sous-groupe de $\Gamma_{K_v}$ correspondant \`a la ramification sauvage. On a la suite exacte de restriction-inflation
\[0\to H^1(D_v,G)\to H^1(K_v,G)\to H^1(S_v,G),\]
o\`u l'on rappelle que le groupe $D_v=\Gamma_{K_v}/S_v$ est le groupe profini engendr\'e par deux g\'en\'erateurs $\sigma_v$ et $\tau_v$ soumis \`a
la relation $\sigma_v\tau_v\sigma_v^{-1}=\tau_v^{q_v}$, o\`u $q_v$ est le cardinal du corps r\'esiduel (cf. \cite[7.5.3]{NSW}). Le groupe
$S_v$ \'etant un pro-$p$-groupe et le $K_v$-groupe $G$ \'etant constant d'exposant $n$ premier \`a $p$, on a que
$H^1(S_v,G)=\hom(S_v,G)/\mathrm{conj.}=1$. Il en va de m\^eme pour $H$, donc il nous suffit d'\'etudier l'application $H^1(D_v,G)\to H^1(D_v,H)$.

La classe $\alpha\in H^1(K_v,H)$ est alors repr\'esent\'ee par un morphisme $a:D_v\to H$ et son noyau d\'efinit une extension totalement ramifi\'ee
et cyclique de $K_v$, ce qui nous dit que l'image de $a$ est un sous-groupe cyclique de $H$.  Or, on sait que $\tau_v$ engendre le sous-groupe
$I_v$ de $D_v$ correspondant \`a $\Gamma_{K^{\mathrm{mr}}/K^\nr}$, o\`u $K^\mathrm{mr}/K$ est l'extension mod\'er\'ement ramifi\'ee maximale de $K$.
Cela nous dit qu'il y a \'equivalence entre affirmer que $\alpha$ est totalement ramifi\'ee et qu'elle est repr\'esent\'ee par un morphisme $a$ dont
l'image est engendr\'ee par $a(\tau_v)$. En effet, l'image d'un tel morphisme est isomorphe au groupe cyclique $\Gamma_{L_v/K_v}$ correspondant \`a
l'extension $L_v/K_v$ induite par le noyau de $a$. La classe $\alpha$ est alors totalement ramifi\'ee si et seulement si cette extension l'est,
ce qui est le cas si et seulement si cette extension n'admet pas de sous-extension non ramifi\'ee, ce qui est vrai pour sa part si et seulement si
l'image $\bar\tau_v$ de $\tau_v$ dans $\Gamma_{L_v/K_v}$ engendre tout le groupe car la sous-extension non ramifi\'ee maximale de $L_v/K_v$ est
celle invariante par $\bar\tau_v$. Ainsi, on d\'eduit que $a(\tau_v)$ engendre l'image de $a$ et l'on voit par la suite que $a(\sigma_v)=a(\tau_v)^m$
pour un certain $m\in\n$.

Soient maintenant $g\in G$ une pr\'eimage de $a(\tau_v)$ dans $G$.
On d\'efinit le morphisme $b:D_v\to G$ en posant $b(\tau_v)=g$ et $b(\sigma_v)=g^m$. C'est un morphisme bien d\'efini car on a
\[b(\sigma_v\tau_v\sigma_v^{-1})=b(\sigma_v)b(\tau_v)b(\sigma_v)^{-1}=b(\tau_v)^mb(\tau_v)b(\tau_v)^{-m}=b(\tau_v)\]
et
\[b(\tau_v^{q_v})=b(\tau_v)^{q_v}=b(\tau_v),\]
car, l'exposant de $G$ \'etant $n$, on sait que $b(\tau_v)^n=1$ et, puisque $K_v$ contient $\zeta_n$, on a que $n$ divise $q_v-1$. On voit alors
imm\'ediatement que la classe $\beta$ de $b$ dans $H^1(K_v,G)$ rel\`eve $\alpha$. Il est aussi \'evident que $\beta$ est totalement ramifi\'ee et
cyclique car l'image de $b$ est \'egale au sous-groupe de $G$ engendr\'e par $g$ et est donc engendr\'ee par $b(\tau_v)$, d'o\`u le noyau de
$b$ d\'efinit une extension cyclique de $K_v$ qui est totalement ramifi\'ee d'apr\`es l'argument donn\'e dans le dernier paragraphe.
\end{proof}

\section{Approximation faible pour des stabilisateurs r\'esolubles}\label{section AF}
D\'esormais on notera $\mu(K)$ le groupe des racines de l'unit\'e contenues dans le corps $K$. On notera aussi $m(K)$ le cardinal de $\mu(K)$.
On rappelle qu'une extension cyclotomique $L/K$ est une extension engendr\'ee par des racines de l'unit\'e.

Dans cette section, on montre le r\'esultat principal de ce texte, qui est le suivant :

\begin{thm}\label{theoreme AF pour les groupes resolubles}
Soient $K$ un corps global et $G$ un $K$-groupe fini r\'esoluble d'exposant $n$ premier \`a la caract\'eristique de $K$. On suppose qu'il existe une
extension finie galoisienne $K'/K$ d\'eployant $G$ telle que $(m(K'),n)=1$. Soit $L/K'$ une extension cyclotomique contenant $\zeta_n$. Soit $S$ un ensemble
fini de places de $K$ contenant les places archim\'ediennes et les places $v$ ramifi\'ees pour $K'/K$. Soit enfin $P$ l'ensemble des places de $K$
totalement d\'ecompos\'ees pour l'extension $L/K$ ne divisant pas $n$ et n'appartenant pas \`a $S$.

Alors, pour tout \'el\'ement $(\beta_v)_{v\in S}$ de $\prod_S H^1(K_v,G)$, il existe une classe $\alpha\in H^1(K,G)$ telle que, si l'on note
$\alpha_v$ l'image de $\alpha$ dans $H^1(K_v,G)$, on a
\begin{itemize}
\item[(i)] pour toute place $v\in S$, $\alpha_v=\beta_v$ ;
\item[(ii)] pour toute place $v\not\in S$, $\alpha_v$ est cyclique et, si $\alpha_v$ est ramifi\'ee, alors elle est totalement ramifi\'ee et $v\in P$.
\end{itemize}
En particulier, si $V=G\backslash G'$ pour un $K$-plongement $G\to G'$ avec $G'$ un $K$-groupe semi-simple simplement connexe, alors $V$ v\'erifie
l'approximation faible.
\end{thm}

Avant de d\'emontrer le r\'esultat, on montre que, d\`es que l'approximation faible (ou encore ``mieux que faible'') est v\'erifi\'ee, on
a un certain contr\^ole sur les cocycles relevant les classes locales. Ce fait, exprim\'e dans le lemme suivant, est d'une importance capitale
dans la d\'emonstration.

\begin{lem}\label{lemme controle des cocycles qui relevent}
Soit $G$ un $K$-groupe fini d'exposant $n$ premier \`a la caract\'eristique de $K$ d\'eploy\'e par une extension galoisienne $K'/K$. Soient
$\ell_1,\ldots,\ell_r$ des nombres premiers diff\'erents de la caract\'eristique de $K$. On suppose que $\zeta_{\ell_i}\not\in K'$ pour tout
$1\leq i\leq r$. Supposons enfin que $G$ v\'erifie l'approximation ``mieux que faible'' par rapport \`a un ensemble $P\subset\Omega_K$, c'est-\`a-dire que,
pour toute donn\'ee $(S,(\beta_v)_{v\in S})$ comme dans le th\'eor\`eme \ref{theoreme AF pour les groupes resolubles}, on peut trouver $\alpha\in H^1(K_v,G)$
v\'erifiant les propri\'et\'es (i) et (ii).

Alors pour toute telle donn\'ee $(S,(\beta_v)_{v\in S})$, il existe une extension galoisienne $K''/K$ v\'erifiant
\begin{itemize}
\item[(*)] $K''$ contient $K'$ et ne contient pas $\zeta_{\ell_i}$ pour tout $1\leq i\leq r$ ;
\end{itemize}
et une classe $\alpha\in H^1(K''/K,G)$ v\'erifiant les propri\'et\'es (i) et (ii).
\end{lem}

\begin{proof}
Soit $v_i$ une place de $K$ n'appartenant pas \`a $S$ totalement d\'ecompos\'ee en $K'$, mais qui reste inerte pour l'extension
$K'(\zeta_{\ell_i})/K'$, extension qui est par ailleurs galoisienne sur $K$. En d'autres termes, le groupe $\Gamma_{K'(\zeta_{\ell_i})/K}$ est une extension
d'un groupe cyclique par le groupe $\Gamma_{K'/K}$ et on demande alors que le groupe de d\'ecomposition de $v_i$ corresponde \`a ce sous-groupe cyclique
distingu\'e de $\Gamma_{K'(\zeta_{\ell_i})/K}$. Une telle place existe en dehors de $S$ d'apr\`es le th\'eor\`eme de \v Cebotarev.

Il suffit alors de consid\'erer un \'el\'ement $\alpha\in H^1(K,G)$ v\'erifiant les propri\'et\'es (i) et (ii) et tel qu'en plus
pour ses localis\'es $\alpha_{v_i}\in H^1(K_{v_i},G)$ on ait $\alpha_{v_i}=0$ pour $1\leq i\leq r$, c'est-\`a-dire qu'on rajoute
les $v_i$ \`a l'ensemble fini $S$, on pose $\beta_{v_i}=0$ et on applique l'approximation ``mieux que faible'' pour l'ensemble $S'$ r\'eunion de $S$ et
des $v_i$. En effet, de cela on d\'eduit que, si l'on note $\alpha'$ la restriction de $\alpha$
\`a $H^1(K',G)$, alors elle est repr\'esent\'ee par un morphisme $a'$ (car $K'$ d\'eploie $G$) qui s'annule en le groupe de d\'ecomposition de $w_i$
pour toute place $w_i$ de $K'$ au-dessus de $v_i$, car $\alpha_{v_i}=0$ entra\^\i ne $\alpha'_{w_i}=0$. Or, puisque $v_i$
est totalement d\'ecompos\'ee en $K'/K$, on a $\Gamma_{K_{v_i}}\subset\Gamma_{K'}$ et $K_{v_i}\cong K'_{w_i}$ et les groupes de d\'ecomposition des
diff\'erents $w_i$ s'identifient aux sous-groupes $\sigma\Gamma_{K_{v_i}}\sigma^{-1}$ pour $\sigma\in\Gamma_{K'/K}$. Ainsi, on voit que $a'$
s'annule en $\sigma\Gamma_{K_{v_i}}\sigma^{-1}$ pour $\sigma\in\Gamma_{K'/K}$.

D'autre part, si l'on note $\Delta'$ le noyau de $a'$, on sait qu'il est distingu\'e dans $\Gamma_{K'}$, mais pas forc\'ement dans $\Gamma_K$. Soit alors
$\Delta''$ l'intersection des $\sigma\Delta'\sigma^{-1}$ pour $\sigma\in\Gamma_{K'/K}$. C'est un sous-groupe distingu\'e de $\Gamma_K$ et d\'efinit alors
une extension finie galoisienne $K''/K$. Or, puisque tous les conjugu\'es de $\Gamma_{K_{v_i}}\subset\Gamma_K$ sont dans $\Delta'$, on voit qu'ils sont
aussi contenus dans $\Delta''=\Gamma_{K''}$. Ainsi, on voit que $v_i$ est totalement d\'ecompos\'ee pour l'extension $K''/K$, d'o\`u l'on d\'eduit que
$\zeta_{\ell_i}\not\in K''$. De plus, il est \'evident que $\alpha$ provient de $H^1(K''/K,G)$, car si l'on choisit un cocycle $a:\Gamma_K\to G(\bar K)$
le repr\'esentant, on voit que $a'$ n'est rien d'autre que la restriction de $a$ \`a $\Gamma_{K'}\subset\Gamma_K$ \`a conjugaison pr\`es et il
s'annule alors en $\Gamma_{K''}\subset\Delta'$.

Enfin, il faut remarquer que, si la propri\'et\'e (ii) n'est v\'erifi\'ee a priori que pour les places $v\not\in S'$ et non pas
pour les places $v\not\in S$, il est clair que pour les places $v\in S'\smallsetminus S$ (i.e. pour les $v_i$) on a bien que $\alpha_v$ est non ramifi\'e
(et donc cyclique) car il est trivial par d\'efinition.
\end{proof}

\begin{rem}
Vue la d\'emonstration de ce lemme, on voit qu'on a un r\'esultat totalement analogue pour un groupe $G$ v\'erifiant l'approximation faible, ou
encore l'approximation tr\`es faible, i.e. l'approximation pour tout ensemble fini $S$ de places de $K$ en dehors d'un ensemble fini $S_0$ fix\'e
auparavant et d\'ependant de $G$. C'est-\`a-dire, par exemple, que si $G$ est un groupe fini d'exposant premier \`a la caract\'eristique de $K$
v\'erifiant l'approximation faible et d\'eploy\'e par une extension galoisienne $K'/K$ telle que $\zeta_{\ell_i}\not\in K'$ pour des nombres premiers
$\ell_1,\ldots,\ell_r$ donn\'es, alors pour tout ensemble fini $S$ de places de $K$ on peut relever une famille d'\'el\'ements $(\beta_v)_{v\in S}$
dans $\prod_S H^1(K_v,G)$ en un \'el\'ement $\alpha\in H^1(K''/K,G)$ avec $K''$ contenant $K'$ et ne contenant pas les $\zeta_{\ell_i}$ pour 
$1\leq i\leq r$.
\end{rem}

\begin{proof}[D\'emonstration du th\'eor\`eme \ref{theoreme AF pour les groupes resolubles}]
On d\'emontre l'\'enonc\'e par r\'ecurrence sur le cardinal de $G$, et on fait cela en trois \'etapes.

\paragraph*{\'Etape 1 : \'Etablissement de la r\'ecurrence} Pour $i\geq 0$, notons $D^{i+1}(G):=[D^{i}(G),D^{i}(G)]$ le $(i+1)$-\`eme d\'eriv\'e de
$D^0(G):=G$. Ce sont des sous-groupes caract\'eristiques de $G$ et, puisque $G$ est r\'esoluble, on sait qu'il existe $j\in\n$ tel que $D^j(G)\neq\{1\}$ et
$D^{j+1}(G)=\{1\}$, d'o\`u l'on d\'eduit que $D^j(G)$ est ab\'elien et caract\'eristique dans $G$. Soit alors $A$ la partie de $\ell$-torsion de $D^j(G)$
pour un $\ell$ donn\'e divisant le cardinal de $D^j(G)$ (et donc de $G$). C'est aussi un sous-groupe caract\'eristique de $G$, car il est caract\'eristique
dans $D^j(G)$. Le sous-groupe $A$ est alors un $\Gamma_K$-module de $\ell$-torsion non trivial, ce qui nous permet de regarder $H=G/A$ comme un
$\Gamma_K$-groupe de cardinal strictement plus petit que celui de $G$.

Pour ces donn\'ees, on a le diagramme commutatif \`a lignes exactes suivant.
\begin{equation}\label{equation diagramme approximation faible}
\xymatrix{
H^1(K,A) \ar[d] \ar[r] & H^1(K,G) \ar[d] \ar[r] & H^1(K,H) \ar[d] \\
\displaystyle{\prod_{v\in \Omega_K} H^1(K_v,A)} \ar[r] & \displaystyle{\prod_{v\in \Omega_K} H^1(K_v,G)} \ar[r] & \displaystyle{\prod_{v\in \Omega_K} H^1(K_v,H)}.
}
\end{equation}
Dans le cas o\`u $G$ serait lui-m\^eme un $\Gamma_K$-module de $\ell$-torsion (cas o\`u l'on aurait par cons\'equent $A=G$), il se peut que $G$ soit
simple ou qu'il ne le soit pas. Dans le premier cas, le r\'esultat se ram\`ene au th\'eor\`eme \ref{theoreme 1 de Neukirch}.
En effet, il suffit de remarquer que, pour $v\in P$ et $G$ de $\ell$-torsion, une classe $\alpha_v\in H^1(K_v,G)$ cyclique et ramifi\'ee est forc\'ement
d\'eploy\'ee par une extension totalement ramifi\'ee, car elle correspond \`a un morphisme (on rappelle que pour $v\in P$ le groupe $G$
est constant en tant que $\Gamma_{K_v}$-groupe) dont l'image est isomorphe \`a $\z/\ell\z$ et elle est alors
d\'eploy\'ee par une extension de degr\'e $\ell$, laquelle est totalement ramifi\'ee d\`es qu'elle est ramifi\'ee.

Dans le deuxi\`eme cas, on (re)d\'efinit $A$ comme n'importe quel sous-module propre non trivial de $G$, ce qui assure que $A$ sera toujours un
sous-$\Gamma_K$-module de cardinal strictement plus petit que celui de $G$ et on peut alors supposer que l'on a le diagramme
\eqref{equation diagramme approximation faible} et que tant $A$ que $H$ v\'erifient l'\'enonc\'e pour la m\^eme extension $L/K$, ainsi que pour
toute autre extension v\'erifiant les m\^emes hypoth\`eses.

\paragraph*{\'Etape 2 : Construction d'un \'el\'ement global dans $H^1(K,G)$} Consid\'erons alors $(\beta_v)_{v\in S}\in\prod_{S}H^1(K_v,G)$. Soit
$(\gamma_v)_{v\in S}$ l'image de cet \'el\'ement dans $\prod_S H^1(K_v,H)$. On peut le relever en un \'el\'ement $\gamma\in H^1(K,H)$ v\'erifiant la
deuxi\`eme propri\'et\'e de l'\'enonc\'e. De plus, le lemme \ref{lemme controle des cocycles qui relevent} nous dit que l'on peut supposer que $\gamma$
provient de $H^1(K''/K,H)$ pour une certaine extension galoisienne $K''/K$ contenant $K'$ et ne contenant pas $\zeta_\ell$ pour tout $\ell$ divisant
$n$. On obtient alors des localisations $\gamma_v\in H^1(K_v,H)$ pour toutes les places $v\not\in S$.

Suivant Serre (cf. \cite[I.5.6]{SerreCohGal}), pour un cocycle $c$ repr\'esentant $\gamma$, on peut associer \`a $\gamma$ une classe
$\delta(\gamma)\in H^2(K,\asd{}{c}{}{}{A})$, o\`u $\asd{}{c}{}{}{A}$ est le tordu de $A$ par le cocycle $c$, i.e. un module \'egal \`a $A$ en tant que
groupe mais avec l'action tordue
\[\asd{\sigma*^c}{}{}{}{a}=c_\sigma\cdot\asd{\sigma}{}{}{}{a},\]
o\`u $c_\sigma\in H$ agit sur $A$ par conjugaison. De plus, toujours d'apr\`es \cite[I.5.6]{SerreCohGal}, on sait que $\gamma$ se rel\`eve en un
\'el\'ement de $H^1(K,G)$ si et seulement si $\delta(\gamma)=0$. Or, cette construction \'etant compatible avec la localisation, on a des classes
$\delta(\gamma_v)\in H^2(K_v,\asd{}{c}{}{}{A})$ qui correspondent \`a $\delta(\gamma)_v$ pour toute place $v$ de $K$. Ces classes se trouvent
\^etre triviales. En effet :

Pour $v\in S$, $\gamma_v$ provient de $\beta_v\in H^1(K_v,G)$, d'o\`u $\delta(\gamma_v)=0$  d'apr\`es ce que l'on vient de dire.

Pour $v\not\in S$, si $\gamma_v\in H^1(K_v^\nr/K_v,H)$, alors son image dans $H^2(K_v,\asd{}{c}{}{}{A})$ est en fait dans
$H^2(K_v^\nr/K_v,\asd{}{c}{}{}{A})$ et ce groupe est trivial car la dimension cohomologique de $\gal(K_v^\nr/K_v)\cong\hat\z$ est $1$. On voit alors
par ailleurs que l'on peut relever $\gamma_v$ en un \'el\'ement $\beta_v\in H^1(K_v^\nr/K_v,G)$.

Sinon, on sait que $\alpha_v$ est totalement ramifi\'ee et cyclique et que $v\in P$, donc $K_v$ d\'eploie $G$ et contient $\zeta_n$
(car $\Gamma_{K_v}\subset\Gamma_{L}$). Le lemme \ref{lemme de relevement local} nous dit alors qu'il existe une classe $\beta_v\in H^1(K_v,G)$
totalement ramifi\'ee et cyclique relevant $\gamma_v$ et donc en particulier on a $\delta(\gamma_v)=0$.\\

En r\'esumant, pour toute place $v$ de $K$, $\gamma_v$ se rel\`eve en un \'el\'ement $\beta_v\in H^1(K_v,G)$ et donc $\delta(\gamma_v)=0$. De
plus, pour toute place $v\not\in S$ telle que $\beta_v$ est ramifi\'e, on a $v\in P$ et la classe $\beta_v$ est totalement ramifi\'ee et cyclique.
Or, puisque $A$ est d\'eploy\'e par $K'$ et que $c$ est un cocycle que l'on peut supposer d\'efini sur
$K''/K$, on sait que le $K$-groupe $\asd{}{c}{}{}{A}$ est d\'eploy\'e par l'extension $K''/K$, laquelle ne contient pas $\zeta_\ell$ pour tout
$\ell$ divisant $n$. Le corollaire \ref{corollaire nullite du sha2 pour les groupes abeliens} nous dit alors que $\Sha^2(K,\asd{}{c}{}{}{A})=0$,
donc $\delta(\gamma)=0$ et alors $\gamma$ provient d'un \'el\'ement $\beta_0\in H^1(K,G)$.

\paragraph*{\'Etape 3 : Correction de l'\'el\'ement global} Puisque $\beta_0$ s'envoie sur $\gamma$, on sait que pour toute place $v$ on a que
$\beta_{0,v}$ et $\beta_v$ ont la m\^eme image dans $H^1(K_v,H)$. Fixons alors un cocycle $b_0$ repr\'esentant $\beta_0$ et tordons le diagramme
\eqref{equation diagramme approximation faible} par
$b_0$. Pour un \'el\'ement $\xi$ dans un ensemble de cohomologie donn\'e du diagramme \eqref{equation diagramme approximation faible}, on note
$\asd{}{b_0}{}{}{\xi}$ son image apr\`es torsion par $b_0$. On a alors que, pour toute place $v$ de $K$, $\asd{}{b_0}{}{v}{\gamma}$ est trivial,
d'o\`u l'on d\'eduit que $\asd{}{b_0}{}{v}{\beta}$ provient d'un \'el\'ement $\alpha'_{v}\in H^1(K_v,\asd{}{b_0}{}{}{A})$.

On remarque que,
quitte \`a changer $b_0$ par un cocycle cohomologue, on peut supposer qu'il rel\`eve le cocycle $c\in Z^1(K''/K,H)$ repr\'esentant $\gamma$.
Ceci ne veut pas dire que l'on a $b_0\in Z^1(K''/K,G)$, mais au moins on sait que la restriction de $b_0$ \`a $\Gamma_{K''}\subset\Gamma_K$
correspond \`a un morphisme (car $K''$ d\'eploie $G$) \`a valeurs dans le $\ell$-groupe $A$. Ceci nous dit que, si l'on note $K'''/K$
l'extension galoisienne minimale d\'eployant $b_0$ (i.e. celle correspondant au sous-groupe de $\Gamma_K$ dont l'image par $b_0$ est $1$), on a
que le degr\'e de $K'''/K''$ est une puissance de $\ell$ car le groupe $\Gamma_{K'''/K''}$ est alors isomorphe \`a un sous-groupe de $A$. On en
d\'eduit que $K'''$ ne contient pas $\zeta_\ell$ car le degr\'e de $K''(\zeta_\ell)/K''$ divise $\ell-1$ et $\zeta_\ell\not\in K''$.
Il est par ailleurs \'evident que $\asd{}{b_0}{}{}{A}$ est un $\Gamma_K$-module d\'eploy\'e par $K'''$.

Soit alors $S'$ la r\'eunion de $S$ avec l'ensemble des places $v$ telles que $\beta_{0,v}$ est ramifi\'e et les places ramifi\'ees pour l'extension
$K'''/K$. L'hypoth\`ese de r\'ecurrence appliqu\'ee au $K$-groupe $\asd{}{b_0}{}{}{A}$ d\'eploy\'e par $K'''$ nous dit alors
que, si l'on note $P'\subset P$ l'ensemble des places dans $P$ totalement d\'ecompos\'ees pour l'extension $L':=LK'''/K$ et n'appartenant pas
\`a $S'$, il existe une classe $\alpha_1\in H^1(K,\asd{}{b_0}{}{}{A})$ telle que
\begin{itemize}
\item pour toute place $v\in S'$, on a $\alpha_{1,v}=\alpha'_{v}$ ;
\item pour toute place $v\not\in S'$, $\alpha_{1,v}$ est cyclique et, si $\alpha_{1,v}$ est ramifi\'ee, alors elle est totalement ramifi\'ee et $v\in P'$.
\end{itemize}
On remarque maintenant que, pour les places $v\in P'$, le cocycle $b_{0,v}$ est trivial. En effet, puisque les places $v\in P'$ sont
totalement d\'ecompos\'ees pour l'extension $K'''/K$, on a $\Gamma_{K_v}\subset\Gamma_{K'''}$ et $b_0$ s'annule en $\Gamma_{K'''}$. On en d\'eduit alors
que $\asd{}{b_0}{}{}{G}$ et $G$ sont isomorphes en tant que $\Gamma_{K_v}$-groupes et que la bijection donn\'e par la torsion par $b_0$ se
r\'eduit \`a l'identit\'e sur $H^1(K_v,G)$, et il en va de m\^eme donc pour $\asd{}{b_0}{}{}{A}$ et $A$.\\

Notons $\beta_1$ l'image de $\alpha_1$ dans $H^1(K,\asd{}{b_0}{}{}{G})$. Il suffit alors pour conclure de d\'efinir l'\'el\'ement $\alpha\in H^1(K,G)$
comme la pr\'eimage de $\beta_1$ par la bijection donn\'ee par la torsion par $b_0$ (i.e. $\asd{}{b_0}{}{}{\alpha}=\beta_1$). Cet \'el\'ement v\'erifie :
\begin{itemize}
\item pour toute place $v\in S'$, on a que $\alpha_v=\beta_v$ car $\asd{}{b_0}{}{v}{\alpha}=\beta_{1,v}$ est l'image de
$\alpha_{1,v}=\alpha'_{v}\in H^1(K_v,\asd{}{b_0}{}{}{A})$ et $\alpha'_{v}$ s'envoie en $\asd{}{b_0}{}{v}{\beta}\in H^1(K_v,\asd{}{b_0}{}{}{G})$ ;
\item pour toute place $v\not\in S'\cup P'$ on a que $\alpha_v\in H^1(K_v^\nr/K,G)$ : en effet, il suffit de remarquer que l'on a
$\asd{}{b_0}{}{v}{\alpha}=\beta_{1,v}\in H^1(K_v^\nr/K,\asd{}{b_0}{}{}{G})$ car il provient de $\alpha_{1,v}\in H^1(K_v^\nr/K,\asd{}{b_0}{}{}{A})$,
et que l'on a tordu par un cocycle repr\'esentant la classe $\beta_{0,v}\in H^1(K_v^\nr/K,G)$ (on rappelle que $S'$ contient par d\'efinition les places
o\`u $\beta_0$ ramifie et que la torsion est compatible avec les morphismes d'inflation) ;
\item pour toute place $v\in P'$, $\alpha_v$ correspond \`a l'image de
$\alpha_{1,v}\in H^1(K_v,\asd{}{b_0}{}{}{A})=H^1(K_v,A)$ car la torsion est r\'eduite \`a l'identit\'e pour ces places ;
\end{itemize}
d'o\`u l'on voit que l'\'el\'ement $\alpha$ convient, car on a alors
\begin{itemize}
\item pour toute place $v\in S$, elle est contenue dans $S'$ et alors $\alpha_v=\beta_v$.
\item pour toute place $v\not\in S$, si $\alpha_v$ ramifie, alors elle est dans $S'\cup P'$, donc soit on a $v\in S'$,
auquel cas on a aussi $\alpha_v=\beta_v$ et alors on a bien $v\in P$ et que $\alpha_v$ est totalement ramifi\'ee et cyclique d'apr\`es
la construction des $\beta_v$ ; soit on a $v\in P'\subset P$ et alors $\alpha_v$ correspond \`a l'image de la classe $\alpha_{1,v}$,
laquelle est bien totalement ramifi\'ee et cyclique.
\end{itemize}
\end{proof}

\section{Principe de Hasse pour des stabilisateurs r\'esolubles}\label{section PH}
Dans cette section on a besoin de la notion de $K$-lien, ou plut\^ot de $\Gamma_K$-lien. On rappelle alors ici sa d\'efinition et quelques propri\'et\'es.

Soient $G$ un groupe abstrait et $\Gamma$ un groupe profini. On note $\aut(G)$ (resp. $\int(G)$) le groupe des automorphismes de $G$
(resp. des automorphismes int\'erieurs). On sait que $\int(G)$ est distingu\'e dans $\aut(G)$. On note $\out(G):=\aut(G)/\int(G)$ et $\pi:\aut(G)\to\out(G)$
la projection canonique. On munit ces trois groupes de la topologie discr\`ete.

\begin{defi}\label{definition g-lien}
Un $\Gamma$-lien sur $G$ est un homomorphisme $\kappa:\Gamma\to\out(G)$ tel qu'il existe une application continue $\phi:\Gamma\to\aut(G)$ relevant $\kappa$, i.e.
tel que $\kappa=\pi\circ\phi$.

S'il existe un tel $\phi$ qui soit en plus un homomorphisme de groupes (ou en d'autres termes, si $G$ est un $\Gamma$-groupe via $\phi$), le $\Gamma$-lien
$\kappa$ est dit trivial.
\end{defi}

\begin{rem}
Lorsque $G$ est fini, tout homomorphisme continu $\kappa:\Gamma\to\out(G)$ admet clairement une telle application $\phi$ le relevant, d'o\`u l'on voit que la
donn\'ee d'un $\Gamma$-lien sur $G$ correspond \`a la donn\'ee du morphisme $\kappa$.
\end{rem}

Pour un $\Gamma$-lien $L=(G,\kappa)$, on peut construire via des cocycles (cf. \cite[1.14]{SpringerH2}) un ensemble de $2$-cohomologie $H^2(\Gamma,L)$.
Cet ensemble classifie les extensions de groupes
\[1\to G\to E\to\Gamma\to 1,\]
telles que l'action ext\'erieure de $\Gamma$ sur $G$ induite par conjugaison dans $E$ correspond au morphisme $\kappa:\Gamma\to\out(G)$ (cf. \cite[1.15]{SpringerH2}).
Dans cet ensemble, on distingue un sous-ensemble, not\'e $N^2(\Gamma,L)$ et dit des classes \emph{neutres}, form\'e des classes correspondant \`a des extensions
scind\'ees. On voit imm\'ediatement que ces classes induisent (au moins) une structure de $\Gamma$-groupe sur $G$ tout simplement par conjugaison dans $E$ via un
scindage de la suite.\\

Soit maintenant $X$ un $K$-espace homog\`ene sous un groupe $G'$ semi-simple simplement connexe, soit $x\in X(\bar K)$ et soit $\bar G$ le stabilisateur de $x$,
que l'on suppose fini et r\'esoluble d'exposant $n$ premier \`a la caract\'eristique de $K$. Suivant Springer \cite[1.20]{SpringerH2}, on peut associer
\`a $(X,x)$ une structure de $\Gamma_K$-lien $L=(\bar G(\bar K),\kappa)$ sur $\bar G(\bar K)$, ainsi qu'une classe $\eta(X)\in H^2(K,L):=H^2(\Gamma_K,L)$.
Toujours d'apr\`es Springer \cite[1.20, 1.27]{SpringerH2} (cf. aussi \cite[\S 7.7]{Borovoi93}), on sait que la classe $\eta(X)\in H^2(K,L)$ est neutre si et
seulement s'il existe un espace principal homog\`ene $P$ sous $G'$ qui est au-dessus de $X$, i.e. qui est muni d'un $K$-morphisme d'espaces homog\`enes
$P\to X$.

\begin{pro}\label{proposition K-point entraine neutre}
Soit $X$ comme ci-dessus et supposons que $X(K)\neq \emptyset$. Alors $\eta(X)\in N^2(K,L)$.
\end{pro}

\begin{proof}
Soit $x\in X(K)$ et soit $P:=G'$ l'espace principal homog\`ene trivial sous $G'$, i.e. l'action de $G'$ sur $P$ \'etant la multiplication \`a droite. On
d\'efinit un $K$-morphisme $P\to X$ en envoyant $g\in P=G'$ dans $x\cdot g\in X$. Il est clair que ce morphisme est \'equivariant avec la $K$-action de $G'$
sur $P$. Le morphisme ainsi d\'efini est alors un morphisme d'espaces homog\`enes, ce qui nous dit que $\eta(X)\in N^2(K,L)$ d'apr\`es ce qui a \'et\'e dit
ci-dessus.
\end{proof}

Si l'on note $\bar Z$ le centre de $\bar G$, on voit clairement que l'on a un morphisme canoniquement induit $\aut(\bar G(\bar K))\to\aut(\bar Z(\bar K))$ qui
passe au quotient en un morphisme $\out(\bar G(\bar K))\to\aut(\bar Z(\bar K))$. Le morphisme $\kappa$ induit alors par composition une structure de
$\Gamma_K$-module sur $\bar Z(\bar K)$, i.e. une $K$-forme $Z$ de $\bar Z$ puisqu'il s'agit d'un groupe fini (cf. aussi \cite[1.16]{SpringerH2}). On peut montrer
facilement que l'ensemble $H^2(K,L)$, d\`es qu'il est non vide, admet une structure d'espace principal homog\`ene sous le groupe $H^2(K,Z)$
(cf. \cite[1.17]{SpringerH2}). De m\^eme, on peut montrer que, d\`es que le sous-ensemble $N^2(K,L)$ est non vide, donc d\`es que le groupe $\bar G(\bar K)$ admet
une structure de $\Gamma_K$-groupe, ou encore, d\`es que $\bar G$ admet une $K$-forme $G$, le sous-ensemble $N^2(K,L)$ peut \^etre d\'ecrit par cette action de
la fa\c con suivante :

\begin{pro}\label{proposition classes neutres}
Soient $\bar G$, $\kappa$, $L$, $\bar Z$ et $Z$ comme ci-dessus. Soient $\eta\in N^2(K,L)$, $\alpha\in H^2(K,L)$ et $\xi\in H^2(K,Z)$ le seul \'el\'ement tel que
$\alpha=\xi\cdot\eta$. Soit enfin $G$ une $K$-forme de $\bar G$ induite par $\eta$. Alors $\alpha\in N^2(K,L)$ si et seulement si $\xi$ appartient \`a l'image
de l'application $H^1(K,G/Z)\xrightarrow{\delta}H^2(K,Z)$ induite par la suite exacte
\[1\to Z\to G\to G/Z \to 1.\]
\end{pro}

On remarquera que la $K$-forme $Z$ de $\bar Z$ est bien un $K$-sous-groupe de la $K$-forme $G$ de $\bar G$. Cela se d\'eduit en effet imm\'ediatement de la
fa\c con dont on a obtenu ces $K$-formes.

\begin{proof}
C'est la proposition 2.3 de \cite{Borovoi93}, \`a cela pr\`es que Borovoi se place sur un corps $K$ de caract\'eristique 0, mais sa d\'emonstration marche en fait
sans aucun changement en caract\'eristique $p$ pour des groupes finis d'ordre premier \`a $p$.
\end{proof}

\begin{rem}
On notera que la notion de $K$-lien, un peu plus technique que celle de $\Gamma_K$-lien et d\'efinie via le groupe de morphismes semi-alg\'ebriques $\saut(\bar G)$ de
$\bar G$, n'est pas n\'ecessaire dans notre cas car on a affaire \`a des groupes finis d'ordre premier \`a la caract\'eristique. En effet, un tel $\bar K$-groupe fini
$\bar G$ admet toujours une $K$-forme canonique (\`a savoir, celle de $K$-groupe constant) et cette forme nous dit que le groupe $\saut(\bar G)$ est dans ce cas
isomorphe au produit direct $\aut(\bar G(\bar K))\times\Gamma_K$ et il en va de m\^eme pour les automorphismes ext\'erieurs. On a alors une \'equivalence entre
$K$-liens sur $\bar G$ et $\Gamma_K$-liens sur $\bar G(\bar K)$.
\end{rem}

Tous ces rappels \'etant faits, on peut passer au r\'esultat principal de cette section, qui est le suivant.

\begin{thm}\label{theoreme PH pour les groupes resolubles}
Soit $X$ un espace homog\`ene sur $G'$ \`a stabilisateur fini r\'esoluble $\bar G$ d'exposant $n$ comme ci-dessus. Soit $K'/K$ l'extension
correspondant au noyau du morphisme $\kappa:\Gamma_K\to \out(\bar G(\bar K))$ correspondant \`a $(X,x)$. On suppose que $(m(K'),n)=1$, o\`u $m(K')$ est
le cardinal de $\mu(K')$. Alors, si $X(K_v)\neq\emptyset$ pour toute place $v\in\Omega_K$, la vari\'et\'e $X$ admet un $K$-point, i.e.
$X(K)\neq\emptyset$.
\end{thm}

\begin{rem}
Si le lien $L$ (et par cons\'equent le morphisme $\kappa$) d\'epend du choix du $\bar K$-point $x$, l'extension $K'$ en est en fait
totalement ind\'ependante. Cela peut se voir facilement \`a partir de la construction de $\kappa$ dans \cite[1.20]{SpringerH2}.
\end{rem}

La d\'emonstration de ce ``principe de Hasse'' utilise de fa\c con cruciale le th\'eor\`eme \ref{theoreme AF pour les groupes resolubles}.

\begin{proof}
On montrera d'abord la neutralit\'e de $\eta(X)$, ce qu'on fera en trois \'etapes.

\paragraph*{\'Etape 1 : Existence d'une classe neutre dans $H^2(K,L)$} Si l'on suppose que $X(K_v)\neq\emptyset$ pour toute place $v$ de $K$, la proposition
\ref{proposition K-point entraine neutre} nous dit que $\eta(X_v)=\eta(X)_v\in H^2(K_v,L)$ est neutre pour toute place $v$ de $K$, o\`u $X_v=X\times_K K_v$.

On rappelle maintenant que l'on a la suite exacte
\[1\to \int (\bar G(\bar K)) \to \aut (\bar G(\bar K)) \xrightarrow{\pi} \out (\bar G(\bar K)) \to 1, \]
et que la neutralit\'e de $\eta(X_v)$ correspond \`a l'existence d'un morphisme $f_v:\Gamma_{K_v}\to \aut(\bar G(\bar K))$ relevant
$\kappa_v:\Gamma_{K_v}\to\out(\bar G(\bar K))$, o\`u $\kappa_v$ est bien entendu la restriction \'evidente de
$\kappa:\Gamma_K\to\out(\bar G(\bar K))$. Les hypoth\`eses faites sur $K'$ et le fait que $\int(\bar G(\bar K))$ soit un groupe r\'esoluble (en effet,
il est isomorphe \`a $\bar G(\bar K)/\bar Z(\bar K)$) nous permettent d'utiliser un th\'eor\`eme de Neukirch, \`a savoir le ``Main Theorem'' de
\cite{NeukirchSolvable}, lequel est valable aussi en caract\'eristique positive, cf. \cite[9.5.5]{NSW}.\footnote{Pour \^etre plus pr\'ecis, on utilise ce
th\'eor\`eme pour la suite exacte \[1\to\int (\bar G(\bar K))\to E\to\im(\kappa)\to 1,\] o\`u $E$ d\'esigne la pr\'eimage de $\im(\kappa)$ par
$\pi:\aut(\bar G(\bar K))\to \out(\bar G(\bar K))$.}
Ce th\'eor\`eme assure l'existence d'un morphisme $f_0:\Gamma_K\to \aut(\bar G(\bar K))$ relevant $\kappa$ et induisant $f_v$ pour un sous-ensemble
fini $S\subset\Omega_K$ que l'on peut prendre \`a discr\'etion. Ce morphisme correspond \`a une classe neutre $\eta_0$ de $H^2(K,L)$ et d\'efinit
une $K$-forme $G_0$ de $\bar G$.

\paragraph*{\'Etape 2 : Choix d'une ``bonne'' classe neutre dans $H^2(K,L)$} Il s'agit maintenant de faire un bon choix de l'ensemble $S$ et des
rel\`evements $f_v$ de $\kappa_v$ pour les places $v\in S$ pour avoir un certain contr\^ole sur la $K$-forme $G_0$, de fa\c con que l'on puisse
appliquer par la suite le th\'eor\`eme principal d'approximation faible (th\'eor\`eme \ref{theoreme AF pour les groupes resolubles}). La d\'emarche
est la m\^eme que dans le lemme \ref{lemme controle des cocycles qui relevent} : consid\`erons, pour
chaque nombre premier $\ell_i$ divisant $n$, une place $v_i$ totalement d\'ecompos\'ee pour l'extension $K'/K$ et inerte pour l'extension
$L'_i=K'(\zeta_{\ell_i})/K'$. On voit alors que $\kappa_{v_i}$ correspond au morphisme trivial car $\Gamma_{K_{v_i}}\subset\Gamma_{K'}$ et alors on
peut supposer que $f_{v_i}$ est aussi trivial. On d\'efinit $S$ comme l'ensemble des places $v_i$ que l'on vient de consid\'erer. 

Consid\'erons alors le morphisme $f_0:\Gamma_K\to \aut(\bar G(\bar K))$ donn\'e par le th\'eor\`eme de Neukirch pour ce choix particulier de $S$
et des morphismes $f_{v}$ pour $v\in S$ et notons $K''/K$ l'extension galoisienne correspondant au noyau de $f_0$. On voit donc que
$\zeta_{\ell_i}\not\in K''$ pour tout $i$. En effet, de m\^eme que dans le lemme \ref{lemme controle des cocycles qui relevent},
on voit que si $K''$ contenait $\zeta_{\ell_i}$, alors on aurait $L'_i\subset K''$, ou encore $\Gamma_{K''}\subset\Gamma_{L'_i}$, d'o\`u l'on aurait
$\Gamma_{K_{v_i}}\not\subset\Gamma_{K''}$ car $\Gamma_{K_{v_i}}\not\subset\Gamma_{L'_i}$ par hypoth\`ese. De cela on voit enfin que la restriction
$f_{v_i}$ de $f_0$ ne serait pas triviale, contredisant nos hypoth\`eses.

Le morphisme $f_0$ que l'on vient de construire correspond alors \`a une classe neutre $\eta_0$ de $H^2(K,L)$ et d\'efinit une $K$-forme $G_0$ de
$\bar G$ d\'eploy\'ee par une extension $K''/K$ ne contenant pas $\zeta_\ell$ pour tout nombre premier $\ell$ divisant $n$. 

\paragraph*{\'Etape 3 : Application du th\'eor\`eme principal}
En rappelant alors que $H^2(K,L)$ admet une structure d'espace principal homog\`ene sous le groupe $H^2(K,Z)$ (cf. le d\'ebut de la section),
on voit qu'il existe une classe $\xi\in H^2(K,Z)$ telle que $\xi\cdot\eta_0=\eta(X)$. La proposition \ref{proposition classes neutres} nous dit que
la neutralit\'e de $\eta(X)$ est \'equivalente au fait que $\xi$ provienne de $H^1(K,G_0/Z)$. Il s'agit alors de trouver une pr\'e-image de $\xi$ dans cet ensemble.

On sait que pour presque toute place $v$ de $K$ (i.e. sauf pour un nombre fini) on a $\xi_v\in H^2(K_v^\nr/K_v,Z)=0$. Soit alors $S_0$ la r\'eunion
des places archim\'ediennes, des places ramifi\'ees pour l'extension $K''/K$ et de l'ensemble fini des places de $K$ telles que $\xi_v\neq 0$. Pour ces
places on sait, toujours d'apr\`es la proposition \ref{proposition classes neutres}, que $\xi_v$ provient d'un \'el\'ement $\psi_v\in H^1(K_v,G_0/Z)$ car
les $\eta(X_v)$ sont tous neutres.

Le th\'eor\`eme \ref{theoreme AF pour les groupes resolubles} nous dit alors que, si l'on note $P$ l'ensemble des places totalement d\'ecompos\'ees
pour l'extension $L''=K''(\zeta_n)/K$, ne divisant pas $n$ et n'appartenant pas \`a $S_0$, il existe alors une classe $\alpha\in H^1(K,G_0/Z)$
telle que
\begin{itemize}
\item pour toute place $v\in S_0$, $\alpha_{v}=\psi_{v}$ ;
\item pour toute place $v\not\in S_0$, $\alpha_v$ est cyclique et, si $\alpha_v$ est ramifi\'ee, alors elle est totalement ramifi\'ee et $v\in P$.
\end{itemize}
On veut montrer que $\alpha$ rel\`eve $\xi$, ce qui donnera par suite que $\eta(X)$ est une classe neutre. Pour ce faire, on donne le m\^eme argument
que dans l'\'etape 2 de la preuve du th\'eor\`eme \ref{theoreme AF pour les groupes resolubles} : soit $v\not\in S_0$, on a alors la suite exacte
\[H^1(K_v,G_0)\to H^1(K_v,G_0/Z)\to H^2(K_v,Z).\]
Si $\alpha_v$ est non ramifi\'e, on sait que son image dans $H^2(K_v,Z)$ est triviale car $H^2(K_v^\nr/K_v,Z)=0$. Sinon, on sait que $v\in P$, donc en
particulier que $\zeta_n\in K_v$ et que $G_0$ est un $K_v$-groupe constant, et que $\alpha_v$ est cyclique et totalement ramifi\'ee. Le lemme
\ref{lemme de relevement local} nous dit alors qu'il existe une classe $\beta_v\in H^1(K_v,G_0)$ relevant $\alpha_v$, ce qui nous dit que
l'image de $\alpha_v$ dans $H^2(K_v,Z)$ est nulle.

En r\'esumant, on voit que pour toute place $v\not\in S_0$, on a que $\alpha_{v}$ provient d'un \'el\'ement dans $H^1(K_v,G_0)$ et alors que son image
dans $H^2(K_v,Z)$ est triviale, d'o\`u l'on voit que l'image de $\alpha_v$ dans $H^2(K_v,Z)$ est \'egale \`a $\xi_v$ pour \emph{toute} place $v\in\Omega_K$.
Enfin, il est facile de voir que $Z$ est un $K$-groupe d\'eploy\'e par $K'/K$, ce qui nous dit qu'il v\'erifie les hypoth\`eses du corollaire
\ref{corollaire nullite du sha2 pour les groupes abeliens} et on a alors que $\Sha^2(K,Z)=0$, donc $\alpha$ est bien un rel\`evement de $\xi$.

\paragraph*{Fin de la d\'emonstration} La proposition \ref{proposition classes neutres} nous permet alors de conclure que $\eta(X)$ est neutre, prouvant
finalement l'existence d'un espace principal homog\`ene $P$ sous $G'$ au-dessus de $X$. Dans le cas de caract\'eristique positive, la nullit\'e de
$H^1(K,G')$ (cf. \cite{HarderHasseCarp}) nous dit alors que $X$ admet un $K$-point tout simplement en poussant le $K$-point de $P=G'$ correspondant
\`a l'\'el\'ement neutre. Dans le cas de caract\'eristique $0$ il faut remarquer sinon que, puisque $n$ est premier \`a $m(K')$, il est notamment impair.
Il est facile alors de voir par un argument de d\'evissage que, pour toute place archim\'edienne $v$ de $K$, l'ensemble $H^1(K_v,G_0)$ est trivial,
d'o\`u l'application
\[H^1(K,G_0)\to\prod_{v\in\Omega_{\infty}}H^1(K_v,G_0),\]
est trivialement surjective, et il en est de m\^eme pour toute $K$-forme de $\bar G$. Ainsi, il suffit de suivre la d\'emonstration de
\cite[Lemma 7.10]{Borovoi93} en utilisant la surjectivit\'e de l'application ci-dessus et le principe de Hasse pour le $H^1$ des groupes semi-simples
simplement connexes cit\'e dans l'introduction pour d\'emontrer de la m\^eme fa\c con l'existence d'un $K$-point sur $X$, ce qui conclut.
\end{proof}

\paragraph*{Remerciements} Je tiens \`a remercier David Harari et Cyril Demarche pour d'importantes discussions, pour l'int\'er\^et avec lequel ils
ont suivi ce travail et pour avoir aid\'e \`a corriger certaines fautes, ainsi que Martin Orr pour m'avoir aid\'e \`a v\'erifier la validit\'e de
certains raisonnements.

\bibliographystyle{alpha}
\bibliography{AF_PH}

\end{document}